\documentclass{amsart}
\usepackage{amsmath}
\usepackage{amsfonts}

\newtheorem{theorem}{Theorem}
\theoremstyle{plain}

\newtheorem{corollary}{Corollary}

\newtheorem{definition}{Definition}

\newtheorem{lemma}{Lemma}

\newtheorem{remark}{Remark}

\numberwithin{equation}{section}

\begin{document}
\title[Weighted Ostrowski Type inequality Involving Integral Means]{A New
Weighted Ostrowski Type inequality Involving Integral Means and Intervals}
\author{A. Qayyum$^{1,2}$}
\address{$^{1}$Department of Mathematics, Universiti teknology Patronas,
Malaysia. $^{2}$Department of Mathematics, University of Hail, Hail 2440,
Saudi Arabia}
\email{atherqayyum@gmail.com}
\author{S. S. Dragomir$^{1,2}$}
\address{$^{1}$Mathematics, College of Engineering \& Science, Victoria
University, PO Box 14428, Melbourne City, MC 8001, Australia.\\
$^{2}$School of Computational and Applied Mathematics, University of the
Witwatersrand, Private Bag 3, Johannesburg 2050, South Africa.}
\email{sever.dragomir@vu.edu.au}
\author{M. Shoaib}
\address{ Department of Mathematics, University of Hail, Hail 2440, Saudi
Arabia}
\email{safridi@gmail.com}
\author{ M. A. Latif}
\address{School of Computational and Applied Mathematics, University of the
Witwatersrand, Private Bag 3, Wits 2050, Johannesburg, South Africa}
\email{m\_amer\_latif@hotmail.com}
\date{Today}
\subjclass[2000]{Primary 65D30; Secondary 65D32}
\keywords{ostrowski inequality, weight function, weighted integral mean}
\thanks{This paper is in final form and no version of it will be submitted
for publication elsewhere.}

\begin{abstract}
The ostrowski inequality expresses bounds on the deviation of a function
from its integral mean. The aim of this paper is to establish a new
inequality using weight function which generalizes the inequalities of
Dragomir, Wang and Cerone .The current article obtains bounds for the
deviation of a function from a combination of integral means over the end
intervals covering the entire interval. A variety of earlier results are
recaptured as particular instances of the current development. Applications
for cumulative distribution function are also discussed.
\end{abstract}

\maketitle

\section{Introduction}

Since A. Ostrowski \cite{14} proved his famous inequality in 1938, many
mathematicians have been working about and around it, in many different
directions and with applications in Numerical analysis and Probability etc.
Several generalizations of the ostrowski integral inequality for mappings of
bounded variation, Lipschitzian, monotonic, absolutely continuous, convex
mappings and n-times differentiable mappings with error estimates for some
special means and for some quadrature rules are considered by many authors.
For recent results and generalizations concerning Ostrowski inequality see%
\cite{2}-\cite{4} and \cite{11}

The first (direct) generalization of Ostrowski's inequality was given by
Milovanovi\'{c} and Pecari\'{c} in \cite{12}. It was for the first time that
Ostrowski-Gr\"{u}ss type inequality was given by Dragomir and Wang \cite{7}.
Cheng gave a sharp version of the mentioned inequality in \cite{3}. Dragomir
and Wang \cite{5}-\cite{8} and Cerone \cite{1} \ \ pointed out a result to
the above . Inspired and motivated by the work of Dragomir and Wang \cite{5}-%
\cite{8} and Cerone \cite{1}, we establish new inequalities, which are more
generalized as compared to previous inequalities developed and discussed in
\cite{5}-\cite{8}. Moreover, our results are in weighted form instead of
previous results which are in non-weighted form.

The approach of Dragomir and Wang \cite{5}-\cite{8} and Cerone \cite{1} for
obtaining the bounds of a particular quadrature rule depended on the peano
kernal while we use weighted peano kernel (see for example \cite{10} and
\cite{15} ) in our findings. This approach not only generalizes the results
of Dragomir and Wang \cite{5}-\cite{8} and Cerone \cite{1}, but also gives
some other interesting inequalities as special cases. Some closely related
new results are also discussed. At the end, we will apply our main result
for cumulative distribution function.

Ostrowski proved the following interesting and useful integral inequality:

\begin{theorem}
Let \ \ $f\ $: $\left[ a,b\right] \rightarrow
\mathbb{R}
$ \ be\ continuous on $\left[ a,b\right] $ and differentiable on $\left(
a,b\right) ,$ whose derivative $f^{\prime }:\left( a,b\right) \rightarrow
\mathbb{R}
$ is bounded on $\ \left( a,b\right) ,$ i.e. $\left\Vert f^{\prime
}\right\Vert _{\infty }=\sup_{t\in \left[ a,b\right] }\left\vert f^{\prime
}\left( t\right) \right\vert <\infty $ then%
\begin{equation}
\left\vert S\left( f;a,b\right) \right\vert \leq \left[ \left( \frac{b-a}{2}%
\right) ^{2}+\left( x-\frac{a+b}{2}\right) ^{2}\right] \frac{M}{b-a}
\tag{1.1}
\end{equation}%
for all $x\in \left[ a,b\right] $. The\ constant$\ \frac{1}{4}\;$is sharp in
the sense that it can not be replaced by a smaller one. Where the functional
$S\left( f;a,b\right) $ represents the deviation of $f\left( x\right) $ from
its integral mean over $\left[ a,b\right] $ and be defined by%
\begin{equation}
S\left( f;a,b\right) =f\left( x\right) -M\left( f;a,b\right) ,  \tag{1.2}
\end{equation}%
and%
\begin{equation}
M\left( f;a,b\right) =\frac{1}{b-a}\int\limits_{a}^{b}f\left( x\right) dx.
\tag{1.3}
\end{equation}
\end{theorem}

In a series of papers, Dragomir and Wang \cite{5}-\cite{8} proved $\left(
1.1\right) $ and other variants for $f^{\text{ }\prime }\in L_{p}\left[ a,b%
\right] $ for $p\geq 1,$ the Lebesgue norms making use of a peano kernel
approach and Montgomery's identity $\left[ 15\right] .$ Montgomery's
identity states that for absolutely continuous mappings $f\ $: $\left[ a,b%
\right] \rightarrow
\mathbb{R}
$%
\begin{equation}
f(x)=\frac{1}{b-a}\int\limits_{a}^{b}f(t)dt+\frac{1}{b-a}\int%
\limits_{a}^{b}P(x,t)f^{\text{ }\prime }(t)dt,  \tag{1.4}
\end{equation}%
where the kernel $p$: $\left[ a,b\right] ^{2}\rightarrow
\mathbb{R}
$ is given by%
\begin{equation*}
P(x,t)=\left\{
\begin{array}{cc}
t-a & \text{if \ \ }a\leq t\leq x\leq b \\
t-b\text{ \ \ } & \text{if \ }a\leq x<t\leq b%
\end{array}%
\right.
\end{equation*}%
If we assume that $f^{\text{ }\prime }\in L_{\infty }\left[ a,b\right] $ and
$\left\Vert f^{\prime }\right\Vert _{\infty }=\underset{t\in \left[ a,b%
\right] }{ess}\left\vert f^{\prime }\left( t\right) \right\vert $ then $M$
in $\left( 1.1\right) $ may be replaced by $\left\Vert f^{\prime
}\right\Vert _{\infty }.$

Dragomir and Wang \cite{5}-\cite{8} utilizing an integration by parts
argument, ostensibly Montgomery's identity $\left( 1.4\right) ,$ obtained%
\begin{eqnarray}
&&\left\vert S\left( f;a,b\right) \right\vert  \label{1.5} \\
&&  \notag \\
&\leq &\left\{
\begin{array}{l}
\frac{1}{b-a}\left[ \left( \frac{b-a}{2}\right) ^{2}+\left( x-\frac{a+b}{2}%
\right) ^{2}\right] \left\Vert f^{\prime }\right\Vert _{\infty },f^{\text{ }%
\prime }\in L_{\infty }\left[ a,b\right] \\
\\
\frac{1}{b-a}\left[ \frac{\left( x-a\right) ^{q+1}+\left( b-x\right) ^{q+1}}{%
q+1}\right] ^{\frac{1}{q}}\left\Vert f^{\prime }\right\Vert _{p},f^{\text{ }%
\prime }\in L_{p}\left[ a,b\right] ,p>1,\frac{1}{p}+\frac{1}{q}=1 \\
\\
\frac{1}{b-a}\left[ \frac{b-a}{2}+\left\vert x-\frac{a+b}{2}\right\vert %
\right] \left\Vert f^{\prime }\right\Vert _{1}%
\end{array}%
\right.  \notag
\end{eqnarray}%
where $f\ $: $\left[ a,b\right] \rightarrow
\mathbb{R}
$ is absolutely continuous\ on $\left[ a,b\right] $ and the constants $\frac{%
1}{4},\left[ \frac{1}{q+1}\right] ^{\frac{1}{q}}$ and $\frac{1}{2}$
respectively sharp.

Cerone \cite{1}, proved the following inequality:

\begin{theorem}
Let \ \ $f\ $: $\left[ a,b\right] \rightarrow
\mathbb{R}
$ \ be\ absolutely continuous mapping and define%
\begin{equation}
\tau \left( x;\alpha ,\beta \right) :=f\left( x\right) -\frac{1}{\alpha
+\beta }\left[ \alpha M\left( f;a,x\right) +\beta M\left( f;x,b\right) %
\right]  \tag{1.6}
\end{equation}%
where%
\begin{equation*}
M\left( f;a,b\right) =\frac{1}{b-a}\int\limits_{a}^{b}f\left( x\right) dx
\end{equation*}%
then%
\begin{eqnarray}
&&\left\vert \tau \left( x;\alpha ,\beta \right) \right\vert  \label{1.7} \\
&&  \notag \\
&\leq &\left\{
\begin{array}{l}
\frac{1}{2\left( \alpha +\beta \right) }\left[ \alpha \left( x-a\right)
+\beta \left( b-x\right) \right] \left\Vert f^{\prime }\right\Vert _{\infty
},f^{\text{ }\prime }\in L_{\infty }\left[ a,b\right] \\
\\
\frac{1}{\left( \alpha +\beta \right) \left( q+1\right) ^{\frac{1}{q}}}\left[
\alpha ^{q}\left( x-a\right) +\beta ^{q}\left( b-x\right) \right] ^{\frac{1}{%
q}}\left\Vert f^{\prime }\right\Vert _{p},f^{\text{ }\prime }\in L_{p}\left[
a,b\right] ,p>1,\frac{1}{p}+\frac{1}{q}=1 \\
\\
\frac{1}{2}\left( 1+\frac{\left\vert \alpha -\beta \right\vert }{\alpha
+\beta }\right) \left\Vert f^{\prime }\right\Vert _{1}%
\end{array}%
\right.  \notag
\end{eqnarray}%
where $\left\Vert h\right\Vert $ are the usual Lebesgue norms for $h\in L%
\left[ a,b\right] $ with%
\begin{equation*}
\left\Vert h\right\Vert _{\infty }:=ess\sup_{t\in \left[ a,b\right]
}\left\vert h\left( t\right) \right\vert <\infty
\end{equation*}%
and%
\begin{equation*}
\left\Vert h\right\Vert _{p}:=\left( \int\limits_{a}^{b}\left\vert h\left(
t\right) \right\vert ^{p}dt\right) ^{\frac{1}{p}},1\leq p\leq \infty
\end{equation*}
\end{theorem}

The current paper obtains bounds on the deviation of a function from
weighted integral means from the end of the interval that cover the whole
interval. The Ostrowski type results are recaptured as special cases.

\section{Main Results}

To establish our main results we first give the following essential
definitins and lemmas.

\begin{definition}
We assume that the weight function (or density) $w:(a,b)\longrightarrow
\lbrack 0,\infty )$ to be non-negative and integrable over its entire domain
and consider%
\begin{equation*}
\int\limits_{a}^{b}w(t)dt<\infty .
\end{equation*}%
The domain of $\ w$\ \ \ may be finite or infinite and may vanish at the
boundary point. We denote the moments%
\begin{eqnarray*}
m(a,b) &=&\overset{b}{\underset{a}{\int }}w(t)dt, \\
N\left( a,b\right) &=&\overset{b}{\underset{a}{\int }}f\left( t\right) w(t)dt
\end{eqnarray*}%
Let the functional $S\left( f,w;a,b\right) $ be defined by%
\begin{equation}
S\left( f,w;a,b\right) =f\left( x\right) -M\left( f,w;a,b\right)  \tag{2.1}
\end{equation}%
where%
\begin{equation}
M\left( f,w;a,b\right) =\frac{1}{\int\limits_{a}^{b}w\left( x\right) dx}%
\int\limits_{a}^{b}f\left( x\right) w\left( x\right) dx  \tag{2.2}
\end{equation}%
The function $S\left( f,w;a,b\right) $ represents the deviation of $f\left(
x\right) $ from its weighted integral mean over $\left[ a,b\right] .$
\end{definition}

We start with the following weighted identity which will be used to obtain
bounds.

\begin{lemma}
Let \ \ $f\ $: $\left[ a,b\right] \rightarrow
\mathbb{R}
$ \ be\ an absolutely continuous mapping. Denoted by $P\left( x,.\right) \ $%
: $\left[ a,b\right] \rightarrow
\mathbb{R}
$ the weighted peano kernel is given by%
\begin{equation}
\rho (x,t)=\left\{
\begin{array}{c}
\frac{\alpha }{\alpha +\beta }\frac{1}{m\left( a,x\right) }m\left(
a,t\right) ,\text{ }a\leq t\leq x \\
\\
\frac{\beta }{\alpha +\beta }\frac{1}{m\left( x,b\right) }m\left( b,t\right)
,\text{ }x<t\leq b%
\end{array}%
\right.  \tag{2.3}
\end{equation}%
where $\alpha ,\beta \in
\mathbb{R}
$ are non negative and not both zero, then the weighted identity%
\begin{equation}
\int\limits_{a}^{b}P(x,t)f^{\text{ }\prime }(t)dt=f\left( x\right) -\frac{1}{%
\alpha +\beta }\left[ \alpha \frac{N\left( a,x\right) }{m\left( a,x\right) }%
+\beta \frac{N\left( x,b\right) }{m\left( x,b\right) }\right]  \tag{2.4}
\end{equation}%
holds.
\end{lemma}

\begin{proof}
From $\left( 2.3\right) $, we have%
\begin{eqnarray*}
\int\limits_{a}^{b}P(x,t)f^{\text{ }\prime }(t)dt &=&\frac{\alpha }{\alpha
+\beta }\frac{1}{m\left( a,x\right) }\int\limits_{a}^{x}\left(
\int\limits_{a}^{t}w(u)du\right) f^{\text{ }\prime }(t)dt \\
&&+\frac{\beta }{\alpha +\beta }\frac{1}{m\left( x,b\right) }%
\int\limits_{x}^{b}\left( \int\limits_{b}^{t}w(u)du\right) f^{\text{ }\prime
}(t)dt
\end{eqnarray*}%
where the integration by parts formula has been utilized on the separate
intervals $\left[ a,x\right] $ and $\left( x,b\right] .$ Simplification of
the expressions readily produces the identity as stated.
\end{proof}

We now give our main result.

\begin{theorem}
Let \ \ $f\ $: $\left[ a,b\right] \rightarrow
\mathbb{R}
$ \ be\ absolutely continuous mapping and define%
\begin{equation}
\tau \left( x,w;\alpha ,\beta \right) :=f\left( x\right) -\frac{1}{\alpha
+\beta }\left[ \alpha M\left( f,w;a,x\right) +\beta M\left( f,w;x,b\right) %
\right]  \tag{2.5}
\end{equation}%
where $M\left( f,w;a,b\right) $ is the weighted integral mean defined in $%
\left( 2.2\right) ,$ then%
\begin{eqnarray}
&&\left\vert \tau \left( x,w;\alpha ,\beta \right) \right\vert  \label{2.6}
\\
&&  \notag \\
&\leq &\left\{
\begin{array}{l}
\frac{1}{2\left( \alpha +\beta \right) }\left[ \frac{\alpha }{m\left(
a,x\right) }\left( x-a\right) ^{2}+\frac{\beta }{m\left( x,b\right) }\left(
b-x\right) ^{2}\right] w\left( x\right) \left\Vert f^{\prime }\right\Vert
_{\infty } \\
\\
\frac{1}{\left( q+1\right) ^{\frac{1}{q}}\left( \alpha +\beta \right) }\left[
\left( \frac{\alpha ^{q}}{m\left( a,x\right) }\left( x-a\right) ^{2}+\frac{%
\beta ^{q}}{m\left( x,b\right) }\left( b-x\right) ^{2}\right) w\left(
x\right) \right] ^{\frac{1}{q}}\left\Vert f^{\prime }\right\Vert _{p} \\
\\
\frac{1}{2}\left( 1+\frac{\left\vert \alpha -\beta \right\vert }{\alpha
+\beta }\right) \left\Vert f^{\prime }\right\Vert _{1}%
\end{array}%
\right.  \notag
\end{eqnarray}
\end{theorem}

\begin{proof}
Taking the modulus of $\left( 2.4\right) $, we have from $\left( 2.5\right) $
and $\left( 2.2\right) $%
\begin{equation}
\left\vert \tau \left( x,w;\alpha ,\beta \right) \right\vert =\left\vert
\int\limits_{a}^{b}P(x,t)f^{\text{ }\prime }(t)dt\right\vert \leq
\int\limits_{a}^{b}\left\vert P(x,t)\right\vert \left\vert f^{\text{ }\prime
}(t)\right\vert dt,  \tag{2.7}
\end{equation}%
where we have used the well known properties of the integral and modulus.

Thus, for $f^{\text{ }\prime }\in L_{\infty }\left[ a,b\right] $ from $%
\left( 2.7\right) $ gives%
\begin{equation*}
\left\vert \tau \left( x,w;\alpha ,\beta \right) \right\vert \leq \left\Vert
f^{\prime }\right\Vert _{\infty }\int\limits_{a}^{b}\left\vert
P(x,t)\right\vert dt,
\end{equation*}%
from which a simple calculation using $\left( 2.3\right) $ gives%
\begin{eqnarray*}
&&\int\limits_{a}^{b}\left\vert P(x,t)\right\vert dt \\
&=&\frac{\alpha }{\alpha +\beta }\frac{1}{m\left( a,x\right) }%
\int\limits_{a}^{x}\left( \int\limits_{a}^{t}w(u)du\right) dt+\frac{\beta }{%
\alpha +\beta }\frac{1}{m\left( x,b\right) }\int\limits_{x}^{b}\left(
\int\limits_{t}^{b}w(u)du\right) dt \\
&=&\left[ \frac{\alpha }{m\left( a,x\right) }\left( x-a\right) ^{2}+\frac{%
\beta }{m\left( x,b\right) }\left( b-x\right) ^{2}\right] \frac{w\left(
x\right) }{2\left( \alpha +\beta \right) }.
\end{eqnarray*}%
Hence the first inequality is obtained.%
\begin{equation*}
\left\vert \tau \left( x,w;\alpha ,\beta \right) \right\vert \leq \frac{1}{%
2\left( \alpha +\beta \right) }\left[ \frac{\alpha }{m\left( a,x\right) }%
\left( x-a\right) ^{2}+\frac{\beta }{m\left( x,b\right) }\left( b-x\right)
^{2}\right] w\left( x\right) \left\Vert f^{\prime }\right\Vert _{\infty }
\end{equation*}%
Further, using H\"{o}lder's integral inequality, from $\left( 2.7\right) $
we have for $f^{\text{ }\prime }\in L_{p}\left[ a,b\right] $
\begin{equation*}
\left\vert \tau \left( x,w;\alpha ,\beta \right) \right\vert \leq \left\Vert
f^{\prime }\right\Vert _{p}\left( \int\limits_{a}^{b}\left\vert
P(x,t)\right\vert ^{q}dt\right) ^{\frac{1}{q}}.
\end{equation*}%
where $\frac{1}{p}+\frac{1}{q}=1$ with $p>1.$ Now%
\begin{eqnarray*}
&&\left( \alpha +\beta \right) \left( \int\limits_{a}^{b}\left\vert
P(x,t)\right\vert ^{q}dt\right) ^{\frac{1}{q}} \\
&=&\left[ \alpha ^{q}\frac{1}{m\left( a,x\right) }\int\limits_{a}^{x}\left(
\int\limits_{a}^{t}w(u)du\right) ^{q}dt+\beta ^{q}\frac{1}{m\left(
x,b\right) }\int\limits_{x}^{b}\left( \int\limits_{t}^{b}w(u)du\right) ^{q}dt%
\right] ^{\frac{1}{q}} \\
&& \\
&=&\frac{1}{\left( q+1\right) ^{\frac{1}{q}}}\left[ \left( \frac{\alpha ^{q}%
}{m\left( a,x\right) }\left( x-a\right) ^{2}+\frac{\beta ^{q}}{m\left(
x,b\right) }\left( b-x\right) ^{2}\right) w\left( x\right) \right] ^{\frac{1%
}{q}}
\end{eqnarray*}%
and so the second inequality is obtained%
\begin{eqnarray*}
&&\left\vert \tau \left( x;\alpha ,\beta \right) \right\vert \\
&\leq &\frac{1}{\left( q+1\right) ^{\frac{1}{q}}\left( \alpha +\beta \right)
}\left[ \left( \frac{\alpha ^{q}}{m\left( a,x\right) }\left( x-a\right) ^{2}+%
\frac{\beta ^{q}}{m\left( x,b\right) }\left( b-x\right) ^{2}\right) w\left(
x\right) \right] ^{\frac{1}{q}}\left\Vert f^{\prime }\right\Vert _{p}
\end{eqnarray*}%
Finally, for $f^{\text{ }\prime }\in L_{1}\left[ a,b\right] $, we have from $%
\left( 2.7\right) $ and using $\left( 2.3\right) $%
\begin{equation*}
\left\vert \tau \left( x,w;\alpha ,\beta \right) \right\vert \leq \underset{%
t\in \left[ a,b\right] }{\sup }\left\vert P(x,t)\right\vert \left\Vert
f^{\prime }\right\Vert _{1}
\end{equation*}%
where%
\begin{eqnarray*}
\left( \alpha +\beta \right) \underset{t\in \left[ a,b\right] }{\sup }%
\left\vert P(x,t)\right\vert &=&\max \left\{ \alpha ,\beta \right\} =\frac{%
\alpha +\beta }{2}+\left\vert \frac{\alpha -\beta }{2}\right\vert \\
&=&\left( 1+\frac{\left\vert \alpha -\beta \right\vert }{\alpha +\beta }%
\right) \frac{\left\Vert f^{\prime }\right\Vert _{1}}{2}
\end{eqnarray*}%
This completes the proof of theorem.
\end{proof}

\begin{remark}
If we put $w\left( x\right) =1,$ in $\left( 2.6\right) ,$ we get cerone's
result $\left( 1.7\right) .$If we put $\alpha =\beta $ and $w\left( x\right)
=1,$in $\left( 2.6\right) ,$we get Dragomir's result $\left( 1.5\right) .$
Similarly, for different weights, we can obtain a variety of results.
\end{remark}

\begin{remark}
It should be noted that from $\left( 2.5\right) $ and $\left( 2.1\right) $%
\begin{equation}
\left( \alpha +\beta \right) \tau \left( x,w;\alpha ,\beta \right) =\alpha
S\left( f,w;a,x\right) +\beta S\left( f,w;x,b\right)  \tag{2.8}
\end{equation}%
From $\left( 1.5\right) $ using the triangle inequality, we obtain%
\begin{eqnarray}
&&\left\vert \left( \alpha +\beta \right) \tau \left( x,w;\alpha ,\beta
\right) \right\vert  \label{2.9} \\
&&  \notag \\
&\leq &\left\{
\begin{array}{l}
\frac{\alpha }{2}\frac{1}{m\left( a,x\right) }\left( x-a\right) ^{2}w\left(
x\right) \left\Vert f^{\prime }\right\Vert _{\infty ,\left[ a,x\right] }+%
\frac{\beta }{2}\frac{1}{m\left( x,b\right) }\left( b-x\right) ^{2}w\left(
x\right) \left\Vert f^{\prime }\right\Vert _{\infty ,\left[ x,b\right] } \\
\\
\alpha \left[ \left( \frac{^{1}}{m\left( a,x\right) \left( q+1\right) }%
\left( x-a\right) ^{2}\right) w\left( x\right) \right] ^{\frac{1}{q}%
}\left\Vert f^{\prime }\right\Vert _{p,\left[ a,x\right] } \\
+\beta \left[ \left( \frac{^{1}}{m\left( x,b\right) \left( q+1\right) }%
\left( b-x\right) ^{2}\right) w\left( x\right) \right] ^{\frac{1}{q}%
}\left\Vert f^{\prime }\right\Vert _{p,\left[ x,b\right] } \\
\\
\alpha \left\Vert f^{\prime }\right\Vert _{1,\left[ a,x\right] }+\beta
\left\Vert f^{\prime }\right\Vert _{1,\left[ x,b\right] }%
\end{array}%
\right.  \notag
\end{eqnarray}%
where for $\left[ c,d\right] \subseteq \left[ a,b\right] $%
\begin{equation*}
\left\Vert h\right\Vert _{p,\left[ c,d\right] }:=\left(
\int\limits_{c}^{d}\left\vert h\left( t\right) \right\vert ^{p}dt\right) ^{%
\frac{1}{p}},p\geq 1
\end{equation*}%
and $\left\Vert h\right\Vert _{\infty ,\left[ c,d\right] }:=ess\underset{%
t\in \left[ c,d\right] }{\sup }\left\vert h\left( t\right) \right\vert $

That is,%
\begin{eqnarray}
&&\left\vert \left( \alpha +\beta \right) \tau \left( x,w;\alpha ,\beta
\right) \right\vert  \label{2.10} \\
&\leq &\left\{
\begin{array}{l}
\left( \frac{\alpha }{m\left( a,x\right) }\left( x-a\right) ^{2}+\frac{\beta
}{m\left( x,b\right) }\left( b-x\right) ^{2}\right) w\left( x\right) \frac{%
\left\Vert f^{\prime }\right\Vert _{\infty }}{2} \\
\\
\alpha \left[ \left( \frac{^{1}}{m\left( a,x\right) \left( q+1\right) }%
\left( x-a\right) ^{2}\right) w\left( x\right) \right] ^{\frac{1}{q}}+\beta %
\left[ \left( \frac{^{1}}{m\left( x,b\right) \left( q+1\right) }\left(
b-x\right) ^{2}\right) w\left( x\right) \right] ^{\frac{1}{q}}\left\Vert
f^{\prime }\right\Vert _{p} \\
\\
\left( \alpha +\beta \right) \left\Vert f^{\prime }\right\Vert _{1}%
\end{array}%
\right.  \notag
\end{eqnarray}%
where the expression $\left( 2.10\right) $ involving the $\left\Vert
.\right\Vert _{p}$ norm is coarser.
\end{remark}

The results of $\left( 2.9\right) $ in which the norms are evaluated over
the two subintervals, although finer, they do require more work.

\begin{remark}
It is possible to reduce the amount of work alluded to in Remark 2, since we
may write%
\begin{eqnarray*}
&&\alpha M\left( f,w;a,x\right) +\beta M\left( f,w;x,b\right) \\
&=&\alpha M\left( f,w;a,x\right) +\frac{\beta }{m\left( x,b\right) }\left[
\overset{b}{\underset{a}{\int }}f\left( u\right) w(u)du-\overset{x}{\underset%
{a}{\int }}f\left( u\right) w(u)du\right] \\
&=&\alpha M\left( f,w;a,x\right) -\frac{\beta }{m\left( x,b\right) }\overset{%
x}{\underset{a}{\int }}f\left( u\right) w(u)du+\frac{\beta }{m\left(
x,b\right) }\overset{b}{\underset{a}{\int }}f\left( u\right) w(u)du \\
&=&\left( \alpha +\beta -\beta \sigma _{w}\left( x\right) \right) M\left(
f,w;a,x\right) +\beta \sigma _{w}\left( x\right) M\left( f,w;a,b\right)
\end{eqnarray*}%
where%
\begin{equation}
\frac{m\left( a,b\right) }{m\left( x,b\right) }=\sigma _{w}\left( x\right)
\tag{2.11}
\end{equation}%
Thus, from $\left( 2.5\right) $,%
\begin{eqnarray}
&&\tau \left( x,w;\alpha ,\beta \right)  \label{2.12} \\
&=&\frac{1}{\alpha +\beta }\left( \frac{\alpha }{x-a}m(a,x)+\frac{\beta }{b-x%
}m(x,b)\right) f\left( x\right)  \notag \\
&&-\left[ \left( 1-\frac{\beta }{\alpha +\beta }\sigma _{w}\left( x\right)
\right) M\left( f,w;a,x\right) +\frac{\beta }{\alpha +\beta }\sigma
_{w}\left( x\right) M\left( f,w;a,b\right) \right]  \notag
\end{eqnarray}%
so that for fixed $\left[ a,b\right] ,$ $M\left( f,w;a,b\right) $ is also
fixed.
\end{remark}

\begin{corollary}
Let the condition of Theorem 3 holds. then%
\begin{eqnarray}
&&\left\vert f\left( x\right) -\frac{1}{2}\left[ M\left( f,w;a,x\right)
+M\left( f,w;x,b\right) \right] \right\vert  \label{2.13} \\
&&  \notag \\
&\leq &\left\{
\begin{array}{l}
\left[ \frac{1}{m\left( a,x\right) }\left( x-a\right) ^{2}+\frac{1}{m\left(
x,b\right) }\left( b-x\right) ^{2}\right] \frac{w\left( x\right) \left\Vert
f^{\prime }\right\Vert _{\infty }}{4} \\
\\
\left[ \left( \frac{1}{m\left( a,x\right) }\left( x-a\right) ^{2}+\frac{1}{%
m\left( x,b\right) }\left( b-x\right) ^{2}\right) w\left( x\right) \right] ^{%
\frac{1}{q}}\frac{\left\Vert f^{\prime }\right\Vert _{p}}{2\left( q+1\right)
^{\frac{1}{q}}} \\
\\
\frac{\left\Vert f^{\prime }\right\Vert _{1}}{2}%
\end{array}%
\right.  \notag
\end{eqnarray}
\end{corollary}

\begin{proof}
The result is readily obtained on allowing $\alpha =\beta $ in $\left(
2.6\right) $ so that the left hand side is $\tau \left( x;\alpha ,\alpha
\right) $ from $\left( 2.5\right) .$
\end{proof}

\begin{corollary}
Let the conditions of Theorem 3 hold. Then%
\begin{eqnarray}
&&\left\vert f\left( \frac{a+b}{2}\right) -\frac{1}{\alpha +\beta }\left[
\alpha M\left( f,w;a,\frac{a+b}{2}\right) +\beta M\left( f,w;\frac{a+b}{2}%
,b\right) \right] \right\vert  \label{2.14} \\
&&  \notag \\
&\leq &\left\{
\begin{array}{l}
\frac{1}{2\left( \alpha +\beta \right) }\left[ \frac{\alpha }{m\left( a,%
\frac{a+b}{2}\right) }\left( \frac{b-a}{2}\right) ^{2}+\frac{\beta }{m\left(
\frac{a+b}{2},b\right) }\left( \frac{b-a}{2}\right) ^{2}\right] w\left(
\frac{a+b}{2}\right) \left\Vert f^{\prime }\right\Vert _{\infty },f^{\text{ }%
\prime }\in L_{\infty }\left[ a,b\right] \\
\\
\frac{1}{\left( q+1\right) ^{\frac{1}{q}}\left( \alpha +\beta \right) }\left[
\left( \frac{\alpha ^{q}}{m\left( a,x\right) }\left( \frac{b-a}{2}\right)
^{2}+\frac{\beta ^{q}}{m\left( x,b\right) }\left( \frac{b-a}{2}\right)
^{2}\right) w\left( \frac{a+b}{2}\right) \right] ^{\frac{1}{q}}\left\Vert
f^{\prime }\right\Vert _{p} \\
\\
,f^{\text{ }\prime }\in L_{p}\left[ a,b\right] ,p>1,\frac{1}{p}+\frac{1}{q}=1
\\
\\
\frac{1}{2}\left( 1+\frac{\left\vert \alpha -\beta \right\vert }{\alpha
+\beta }\right) \left\Vert f^{\prime }\right\Vert _{1}%
\end{array}%
\right.  \notag
\end{eqnarray}
\end{corollary}

\begin{proof}
Placing $x=$ $\frac{a+b}{2}$ in $\left( 2.5\right) $ and $\left( 2.6\right) $
produces the results stated in $\left( 2.14\right) .$
\end{proof}

\begin{corollary}
If $\left( 2.13\right) $ is evaluated at the midpoint then%
\begin{eqnarray}
&&\left\vert f\left( \frac{a+b}{2}\right) -\frac{1}{2}M\left( f,w;a,b\right)
\right\vert  \label{2.15} \\
&&  \notag \\
&\leq &\left\{
\begin{array}{l}
\frac{1}{2\left( \alpha +\beta \right) }\left[ \frac{\alpha }{m\left( a,%
\frac{a+b}{2}\right) }\left( \frac{b-a}{2}\right) ^{2}+\frac{\beta }{m\left(
\frac{a+b}{2},b\right) }\left( \frac{b-a}{2}\right) ^{2}\right] w\left(
\frac{a+b}{2}\right) \left\Vert f^{\prime }\right\Vert _{\infty } \\
\\
\frac{1}{\left( q+1\right) ^{\frac{1}{q}}\left( \alpha +\beta \right) }\left[
\left( \frac{\alpha ^{q}}{m\left( a,x\right) }\left( \frac{b-a}{2}\right)
^{2}+\frac{\beta ^{q}}{m\left( x,b\right) }\left( \frac{b-a}{2}\right)
^{2}\right) w\left( \frac{a+b}{2}\right) \right] ^{\frac{1}{q}}\left\Vert
f^{\prime }\right\Vert _{p} \\
\\
\frac{1}{2}\left\Vert f^{\prime }\right\Vert _{1}%
\end{array}%
\right.  \notag
\end{eqnarray}%
which is in agreement with $\left( 1.5\right) $ when $x=$ $\frac{a+b}{2}.$
The above result can also be obtained by taking $\alpha =\beta $ in $\left(
2.14\right) $ or equivalently $\alpha =\beta $ and $x=$ $\frac{a+b}{2}$ in $%
\left( 2.6\right) .$
\end{corollary}

\section{\textbf{An Application to the Weighted Cumulative Distribution
Function}}

Let $X$ be a random variable taking values in the finite interval $\left[ a,b%
\right] $ with Cumulative Distributive Function%
\begin{equation*}
F_{w}\left( x\right) =P_{r}\left( X\leq x\right) =\overset{x}{\underset{a}{%
\int }}f\left( u\right) w\left( u\right) du,
\end{equation*}%
we also use the fact that%
\begin{equation*}
\overset{b}{\underset{a}{\int }}f\left( u\right) w\left( u\right) du=1
\end{equation*}%
where $f$ is a Probability Density Function. The following theorem holds.

\begin{theorem}
Let $X$ and $F$ be as above, then%
\begin{eqnarray}
&&\left\vert \left[ \alpha m\left( x,b\right) -\beta m\left( a,x\right) %
\right] F_{w}\left( x\right) -m\left( a,x\right) \left[ \left( \alpha +\beta
\right) m\left( x,b\right) f\left( x\right) -\beta \right] \right\vert
\label{3.1} \\
&&  \notag \\
&\leq &\left\{
\begin{array}{l}
\frac{1}{2}\left[ \alpha m\left( x,b\right) \left( x-a\right) ^{2}+\beta
m\left( a,x\right) \left( b-x\right) ^{2}\right] w\left( x\right) \left\Vert
f^{\prime }\right\Vert _{\infty },f^{\text{ }\prime }\in L_{\infty }\left[
a,b\right] \\
\\
\frac{1}{\left( q+1\right) ^{\frac{1}{q}}}m\left( x,b\right) m\left(
a,x\right) \left[ \left( \frac{\alpha ^{q}}{m\left( a,x\right) }\left(
x-a\right) ^{2}+\frac{\beta ^{q}}{m\left( x,b\right) }\left( b-x\right)
^{2}\right) w\left( x\right) \right] ^{\frac{1}{q}}\left\Vert f^{\prime
}\right\Vert _{p} \\
\\
,f^{\text{ }\prime }\in L_{p}\left[ a,b\right] ,p>1,\frac{1}{p}+\frac{1}{q}=1
\\
\\
\frac{1}{2}m\left( x,b\right) m\left( a,x\right) \left( \alpha +\beta
+\left\vert \alpha -\beta \right\vert \right) \left\Vert f^{\prime
}\right\Vert _{1},f^{\text{ }\prime }\in L_{1}\left[ a,b\right]%
\end{array}%
\right.  \notag
\end{eqnarray}
\end{theorem}

\begin{proof}
From $\left( 2.5\right) ,$ we have%
\begin{eqnarray*}
&&\tau \left( x,w;\alpha ,\beta \right) \\
&:&=f\left( x\right) -\frac{1}{\alpha +\beta }\left[ \alpha M\left(
f,w;a,x\right) +\beta M\left( f,w;x,b\right) \right] \\
&=&f\left( x\right) -\frac{1}{\alpha +\beta }\left[ \frac{\alpha }{m\left(
a,x\right) }\overset{x}{\underset{a}{\int }}f\left( u\right) w\left(
u\right) du+\frac{\beta }{m\left( x,b\right) }\overset{b}{\underset{x}{\int }%
}f\left( u\right) w\left( u\right) du\right] .
\end{eqnarray*}%
Or%
\begin{eqnarray*}
&&-\left( \alpha +\beta \right) m\left( a,x\right) m\left( x,b\right) \tau
\left( x,w;\alpha ,\beta \right) \\
&=&\left( \alpha m\left( x,b\right) -\beta m\left( a,x\right) \right)
F_{w}\left( x\right) -\left( \alpha +\beta \right) m\left( a,x\right)
m\left( x,b\right) f\left( x\right) +\beta m\left( a,x\right) \\
&& \\
&=&\left( \alpha m\left( x,b\right) -\beta m\left( a,x\right) \right)
F_{w}\left( x\right) -m\left( a,x\right) \left[ \left( \alpha +\beta \right)
m\left( x,b\right) f\left( x\right) -\beta \right] .
\end{eqnarray*}%
The proof follows in a straightforward manner from $\left( 2.6\right) $.

Using $\left( 2.12\right) $ for $\tau \left( x,w;\alpha ,\beta \right) $ and
$\left( 2.13\right) .$Taking the modulus and using $\left( 2.6\right) $
gives the stated result.
\end{proof}

\begin{corollary}
Let $X$ be a random variable, $F\left( w,x\right) $ weighted Cumulative
Distributive Function and $f$ is a Probability Density Function. Then%
\begin{eqnarray}
&&\left\vert \frac{1}{2}\left[ m\left( x,b\right) -m\left( a,x\right) \right]
F_{w}\left( x\right) -m\left( a,x\right) \left[ m\left( x,b\right) f\left(
x\right) -\frac{1}{2}\right] \right\vert  \label{3.2} \\
&&  \notag \\
&\leq &\left\{
\begin{array}{l}
\left[ \frac{1}{4}m\left( x,b\right) \left( x-a\right) ^{2}+m\left(
a,x\right) \left( b-x\right) ^{2}\right] w\left( x\right) \left\Vert
f^{\prime }\right\Vert _{\infty },f^{\text{ }\prime }\in L_{\infty }\left[
a,b\right] \\
\\
m\left( x,b\right) m\left( a,x\right) \frac{1}{2\left( q+1\right) ^{\frac{1}{%
q}}}\left[ \left( \frac{1}{m\left( a,x\right) }\left( x-a\right) ^{2}+\frac{1%
}{m\left( x,b\right) }\left( b-x\right) ^{2}\right) w\left( x\right) \right]
^{\frac{1}{q}} \\
\\
\times \left\Vert f^{\prime }\right\Vert _{p},,f^{\text{ }\prime }\in L_{p}%
\left[ a,b\right] ,p>1,\frac{1}{p}+\frac{1}{q}=1 \\
\\
\frac{1}{2}m\left( x,b\right) m\left( a,x\right) \left\Vert f^{\prime
}\right\Vert _{1},f^{\text{ }\prime }\in L_{1}\left[ a,b\right]%
\end{array}%
\right.  \notag
\end{eqnarray}
\end{corollary}

\begin{remark}
The above result allow the approximation of $F\left( x\right) $ in terms of $%
f\left( x\right) $. The approximation of%
\begin{equation*}
R_{w}\left( x\right) =1-F_{w}\left( x\right)
\end{equation*}%
could also be obtained by a simple substitution. $R_{w}\left( x\right) $ is
of importance in reliability theory where $f\left( x\right) $ is the p.d.f
of failure.
\end{remark}

\begin{remark}
We may directly put $\beta =0$ in $\left( 2.5\right) $ and $\left(
2.6\right) $ , assuming that $\alpha \neq 0$ to obtain%
\begin{eqnarray}
&&\left( \frac{1}{x-a}m(a,x)\right) f\left( x\right) -F_{w}\left( x\right)
\label{3.3} \\
&\leq &\left\{
\begin{array}{l}
\frac{\left( x-a\right) ^{2}}{2}w\left( x\right) \left\Vert f^{\prime
}\right\Vert _{\infty } \\
\\
\left( x-a\right) ^{1+\frac{1}{q}}w\left( x\right) \left\Vert f^{\prime
}\right\Vert _{p}\frac{1}{\left( q+1\right) ^{\frac{1}{q}}} \\
\\
\left( x-a\right) \left\Vert f^{\prime }\right\Vert _{1}%
\end{array}%
\right.  \notag
\end{eqnarray}%
which agrees with $\left( 1.5\right) $ for $\left\vert S\left( f;a,x\right)
\right\vert .$
\end{remark}

We may replace $f$ \ by $F$ in any of the equations $\left( 3.1\right)
-\left( 3.3\right) $ so that the bounds are in terms of $\left\Vert
f^{\prime }\right\Vert _{p},$ $p\geq 1.$ Further we note that%
\begin{equation*}
\overset{b}{\underset{a}{\int }}F_{w}\left( u\right) du=\left. uF_{w}\left(
u\right) \right\vert _{a}^{b}-\overset{b}{\underset{a}{\int }}xw\left(
x\right) f\left( x\right) dx=b-E\left[ Xw\left( X\right) \right] .
\end{equation*}%
\textbf{Conclusion:} Cerone \cite{1}, obtained bounds for the deviation of
a function from a combination of integral means over the end intervals
covering the entire interval and applied these results to approximate the
cumulative distribution function in terms of the probability density
function. Inspired and motivated by the work of Cerone \cite{1}, we
establish a new inequality, which is more generalized as compared to
previous inequalities developed in \cite{5}-\cite{8}. In addition, the
approach of Cerone \cite{1} for obtaining the bounds of a particular
quadrature rule has depended on the peano kernel but we use weighted peano
kernel in our findings. This approach not only generalizes the results of 
\cite{1} but also gives some other interesting inequalities as special
cases. Approximation of the weighted cumulative distribution functions in
terms of weighted probability density function is given as well.

\bigskip

\end{document}